\documentclass[12pt]{article}
\usepackage{amsfonts}
\usepackage{amssymb}
\usepackage{mathrsfs}
\usepackage{srcltx}
\usepackage{amsmath}
\usepackage{amsthm}
\usepackage{amstext}
\usepackage{amsopn}
\usepackage{graphicx}
\usepackage{color}
\newtheorem{theorem}{Theorem}[section]
\newtheorem{lemma}[theorem]{Lemma}
\newtheorem{proposition}[theorem]{Proposition}

\theoremstyle{definition}

\theoremstyle{remark}
\newtheorem{remark}[theorem]{Remark}

\numberwithin{equation}{section}

\newcommand{\ba}{\begin{array}}
\newcommand{\ea}{\end{array}}
\newcommand{\f}{\frac}

\newcommand{\la}{\lambda}

\newcommand{\ds}{\displaystyle}

\begin{document}
\date{}
\title{ \bf\large{Global dynamics of the diffusive Lotka-Volterra competition model with stage structure}\thanks{S. Chen is supported by National Natural Science Foundation of China (No 11771109) and a grant from China Scholarship Council, and J. Shi is supported by US-NSF grant DMS-1715651.}}
\author{ Shanshan Chen\textsuperscript{1}\footnote{Corresponding Author, Email: chenss@hit.edu.cn}\, \, Junping Shi\textsuperscript{2}\footnote{Email: jxshix@wm.edu}\ \
 \\
 {\small \textsuperscript{1} Department of Mathematics, Harbin Institute of Technology,\hfill{\ }}\\
\ \ {\small Weihai, Shandong, 264209, P.R.China.\hfill{\ }}\\
{\small \textsuperscript{2} Department of Mathematics, College of William and Mary,\hfill{\ }}\\
\ \ {\small Williamsburg, Virginia, 23187-8795, USA.\hfill {\ }}\\
}
\maketitle

\begin{abstract}
{The global asymptotic behavior of the classical diffusive Lotka-Volterra competition model with stage structure is studied. A complete classification of the global dynamics is given for the weak competition case. It is shown that under otherwise same conditions, the species with shorter maturation time prevails. The method is also applied to the global dynamics of another delayed competition models.
}

\noindent {\bf{Keywords}}: Reaction-diffusion; Lotka-Volterra Competition model; \\ Global stability; Maturation delay.\\
\noindent {\bf {MSC 2010}}: 35K57, 35K51, 37N25, 92D25
\end{abstract}

\section{Introduction}

The competition for natural resource regulates the growth of biological populations, and it leads to density dependent and bounded population growth. Moreover two species competing for the same limiting resource often cannot coexist, which is the phenomenon of competition exclusion \cite{Gause,Tilman}.
Lotka-Volterra model has been used to describe the competition for resource, and it predicts the competition exclusion to occur in the weak competition case \cite{Lotka,Volterra}. On the other hand, spatial heterogeneity of the environment can change or determine the outcome of the competition, and the dynamical behaviors of spatially explicit mathematical models could explain, to certain extent, the ecological complexity of ecosystems \cite{Levin}.

One of the prototypical mathematical models to describe competition for resource in  spatially heterogeneous environment  is the following diffusive Lotka-Volterra competition system:
\begin{equation}\label{lv}
\begin{cases}
  \ds\frac{\partial U}{\partial t}=d_1\Delta U+U\left(m_1(x)-U-cV\right), & x\in \Omega,\; t>0,\\
 \ds\frac{\partial V}{\partial t}=d_2\Delta V+V\left(m_2(x)-bU-V\right), & x\in\Omega,\; t>0,\\
\ds\f{\partial U}{\partial n}=\ds\f{\partial V}{\partial n}=0,& x\in\partial \Omega,\ t>0,\\
U(x,0)=U_0(x)\ge0,\;V(x,0)=V_0(x)\ge0,&x\in\Omega.
\end{cases}
\end{equation}
Here $U(x,t)$ and $V(x,t)$ are the population densities of two competing species at location $x$ and time $t$ respectively; $\Omega$ is a bounded domain in $\mathbb{R}^N$ with a smooth boundary $\partial \Omega$, and $n$ is the outward unit normal vector on $\partial \Omega$; $d_1,d_2>0$ are the diffusion coefficients of species $U$ and $V$, respectively; 
the functions $m_1(x)$ and $m_2(x)$ represent the intrinsic growth rates of species $U$ and $V$ at location $x$ respectively, and they can also be interpreted as resource available to $U$ and $V$; and the parameters $b,c>0$ account for the inter-specific competition. 
The no-flux boundary conditions are imposed on the boundary, which means that the habitat is closed and individuals cannot move in or out through the boundary.

It is well known that  that system \eqref{lv} has only two semitrivial steady states $\left(\theta_{d_1,m_1},0\right)$ and $\left(0,\theta_{d_2,m_2}\right)$ (see \cite{CantrellB})
under the following assumption:
\begin{enumerate}
\item[$\mathbf{(M^+)}$] $m_i(x)\in C^{\alpha}(\overline \Omega)$, for $\alpha\in(0,1)$, and $m_i(x)>0$ on $\overline\Omega$ for $i=1,2$.
\end{enumerate}
 For the special case that $m_1(x)=m_2(x)=m(x)>0(\not\equiv const)$,
the results on model \eqref{lv} could be summarized as follows.
If $b=c=1$, Dockey et al. \cite{Dockery} showed that
the semitrivial steady state $\left(\theta_{d_1,m_1},0\right)$ is globally asymptotically stable if $d_1<d_2$. That is,
\lq\lq the slower diffuser always wins\rq\rq. If $b,c<1$ (the weak competition case), Lou \cite{Lou} showed that in a parameter region of  $(b,c)$,   the semitrivial steady state $(\theta_{d_1,m_1},0)$ is globally asymptotically stable if it is linearly stable. Lam and Ni \cite{Lam}  showed that in a more genreal parameter region of $(b,c)$, either $(\theta_{d_1,m_1},0)$ is globally asymptotically stable, or \eqref{lv} has a unique coexistence steady state which is globally asymptotically stable.
Finally He and Ni \cite{HeNi2016-1}   gave a complete classification on the global dynamics of model \eqref{lv} for the parameter region satisfying $0<bc\le1$ and all $d_1,d_2>0$: either one of the two semitrivial steady states is globally asymptotically stable, or there exists a unique positive steady state which is globally asymptotically stable, or there exists a compact global attractor which consists of a continuum of steady states. Their results also hold for the case that $m_1(x)\not\equiv m_2(x)$ (see Section 2 for more precise results).
We remark that the results on the dynamics of model \eqref{lv} could also be found in \cite{HeNi2013-1,HeNi2013-2,HeNi2016-2,HeNi2017}, and see \cite{LouY2014,ZhouLX2016,LouZhaoZhou2019,LouY2015,ZhouP2017,ZhouP2016,ZhouP2019,ZhouPZ2019} for the dynamics of competition models in the advective environment.

For some biological species, the time from the birth to maturation may have important effect on the population dynamics, and it should be included in the modeling process.
Considering the maturation time of species $U$ and $V$, we propose the following
diffusive Lotka-Volterra competition model with time-delays:
\begin{equation}\label{lvs}
\begin{cases}
  \ds\frac{\partial U}{\partial t}=d_1\Delta U+e^{-\gamma_1\tau_1}m_1(x)U(x,t-\tau_1)-U^2-cUV, & x\in \Omega,\; t>0,\\
 \ds\frac{\partial V}{\partial t}=d_2\Delta V+e^{-\gamma_2\tau_2} m_2(x)V(x,t-\tau_2)-bUV-V^2, & x\in\Omega,\; t>0,\\
\ds\f{\partial U}{\partial n}=\ds\f{\partial V}{\partial n}=0,& x\in\partial \Omega,\ t>0,\\
U(x,t)=U_0(x,t)\ge0, &x\in\Omega,\;t\in [-\tau_1,0],\\
V(x,t)=V_0(x,t)\ge0,&x\in\Omega,\;t\in [-\tau_2,0].
\end{cases}
\end{equation}
Here $\tau_1$ and $\tau_2$ represent the maturation periods of $U$ and $V$, respectively, $\gamma_1$ and $\gamma_2$ represent the death rates of the immature species of $U$ and $V$, respectively, and other parameters have the same meanings as those in model \eqref{lv}. We remark that if $\tau_1=\tau_2=0$, then model \eqref{lvs} is reduced to \eqref{lv}.
Indeed in \cite{Omari,Al-Omari2003}, a similar model was constructed for $\Omega=(-\infty,\infty)$ and $m_i(x)$ are constant for $i=1,2$, and they studied the existence of the traveling wave front solutions.

The derivation of model \eqref{lvs} starts from the standard age-structured population model (see \cite{Omari,Metz,So2001}), and the details for the unbounded domain could be found in \cite{Omari}. Here we include it for the sake of completeness.
Let $u(x,t,a)$ be the density of a species of age $a$ at space $x$ and time $t$, and $\tau$ be the maturation period. Assume that
$u$ satisfies the age-structured population model:
\begin{equation*}
\begin{cases}
\ds\f{\partial u}{\partial t}+\ds\f{\partial u}{\partial a}=\tilde d \f{\partial^2 u}{\partial x^2}-\gamma u,&x\in\Omega,\;t>0,\;0<a<\tau,\\
\ds\f{\partial u}{\partial n}=0,&x\in\partial\Omega,\;t>0,\;0<a<\tau,\\
\end{cases}
\end{equation*}
and the mature species
$u_m(x,t):=\ds\int_{\tau}^\infty u(x,t,a)da$
satisfies
\begin{equation*}
\begin{cases}
\ds\f{\partial u_m}{\partial t}=d\ds\f{\partial ^2 u_m}{\partial x^2}+u(x,t,\tau)-u_m^2, &x\in\Omega,\;t>0,\\
\ds\f{\partial u_m}{\partial n}=0,&x\in\partial\Omega,\;t>0,
\end{cases}
\end{equation*}
with $u(x,t,0)=m(x)u_m(x,t)$. Here $d$ and $\tilde d$ are the diffusion coefficients of the mature and immature species, respectively, $\gamma$ is the mortality rate of the immature species, $m(x)$ is the intrinsic growth rate of the mature species at space $x$, and $u(x,t,\tau)$ is the mature adult recruitment term. Then
$$u(x,t,\tau)=e^{-\gamma\tau}\int_\Omega G(x,y,\tilde d,\tau) m(y)u_m(y,t-\tau)dy,$$
where the Green's function $G(x,y,\tilde d,t)$
satisfies
\begin{equation*}
\begin{cases}
\ds\f{\partial G}{\partial t}=\tilde d \ds\f{\partial^2 G}{\partial y^2},\;\;\;&y\in\Omega,\;t>0,\\
\ds\f{\partial G}{\partial n}=0,\;\;\;&y\in\partial\Omega,\;t>0,\\
G(x,y,\tilde d,0)=\delta (x-y),\;\;\;&y\in \Omega.
\end{cases}
\end{equation*}
For one dimensional domain $\Omega=(0,L)$, one can calculate that
\begin{equation}\label{oned}
G(x,y,\tilde d,t)=\ds\f{1}{L}+\ds\f{2}{L}\sum_{n=1}^\infty e^{-\f{n^2\pi^2}{L^2}\tilde d t}\cos\ds\f{n\pi x}{L}\cos\ds\f{n\pi y}{L}.
\end{equation}
Consequently, the mature species $u_m(x,t)$ satisfies
\begin{equation*}
\begin{cases}
\ds\f{\partial u_m}{\partial t}=d\ds\f{\partial ^2 u_m}{\partial x^2}+e^{-\gamma\tau}\int_\Omega G(x,y,\tilde d,\tau) m(y)u_m(y,t-\tau)dy-u_m^2, &x\in\Omega,\;t>0,\\
\ds\f{\partial u_m}{\partial n}=0,&x\in\partial\Omega,\;t>0.
\end{cases}
\end{equation*}
Then for two competing species $U$ and $V$, one could obtain the following two species competing model with age structure£º
\begin{equation}\label{lvsnonlocal}
\begin{cases}
  \ds\frac{\partial U}{\partial t}=d_1\Delta U+e^{-\gamma_1\tau_1}\int_\Omega G(x,y,\tilde d_1,\tau_1) m_1(y)U(y,t-\tau_1)dy &\\
  \hspace{0.4in}-U^2-cUV, \; &x\in \Omega,\; t>0,\\
 \ds\frac{\partial V}{\partial t}=d_2\Delta V+e^{-\gamma_2\tau_2}\int_\Omega G(x,y,\tilde d_2,\tau_2) m_2(y)V(y,t-\tau_2)dy &\\
 \hspace{0.4in}-bUV-V^2, \; &x\in\Omega,\; t>0,\\
\ds\f{\partial U}{\partial n}=\ds\f{\partial V}{\partial n}=0,\;\; &x\in\partial \Omega,\ t>0,\\
U(x,t)=U_0(x,t)\ge0, \;&x\in\Omega,\;t\in [-\tau_1,0],\\
V(x,t)=V_0(x,t)\ge0, \;&x\in\Omega,\;t\in [-\tau_2,0].
\end{cases}
\end{equation}
Note that $G(x,y,\tilde d,\tau)=\delta(x-y)$ for $\tilde d=0$, see Eq. \eqref{oned} for the one dimensional case. Therefore, model \eqref{lvsnonlocal} can be approximated by model \eqref{lvs} if the diffusion rates of the immature species of $U$ and $V$ are small ($\tilde d_1$ and $\tilde d_2$ are small).

In \cite{GuoY2018}, Yan and Guo considered the dynamics of the competition model with stage structure and spatial heterogeneity, and investigated model \eqref{lvs} for the case of $\tau_1>0$ and $\tau_2=0$. They showed that one of the semitrivial steady states can be globally asymptotically stable under certain conditions, and the global stability of the positive steady state could be obtained if there exists a pair of upper and lower solutions.
In this paper, we show that, for $0<bc\le1$, the global dynamics of model \eqref{lvs} can be completely classified as the non-delay case \cite{HeNi2016-1}: either one of the two semitrivial steady states is globally asymptotically stable, or there exists a unique positive steady state which is globally asymptotically stable, or there exists a compact global attractor which consists of a continuum of steady states.

The rest of the paper is organized as follows.
In Section 2, we give some preliminaries. In Section 3, we obtain the global dynamics of model \eqref{lvs} for $0<bc\le1$.
In Section 4, we apply the obtained results in Section 3 to two concrete examples and show the effect of time delays.
Moreover, we find that the method for model \eqref{lvs} can also be applied to another delayed competition model.
Throughout the paper, we denote
\begin{equation*}
\Gamma=\{(d_1,d_2,\tau_1,\tau_2,\gamma_1,\gamma_2)\in\left(\mathbb{R}\right)^6:d_1,d_2>0,\tau_1,\tau_2,\gamma_1,\gamma_2\ge0\},
\end{equation*}
$Y=C(\overline \Omega,\mathbb{R})$, $E_i=C([-\tau_i,0],Y)(i=1,2)$, and $E=E_1\times E_2$. Here $E_i=Y$ if $\tau_i=0(i=1,2)$.
Moreover, we denote $\mathbb{R}^+=\{x\in\mathbb{R}:x\ge0\}$, $Y^+=C(\overline \Omega,\mathbb{R}^+)$, $E_i^+=C([-\tau_i,0],Y^+)(i=1,2)$,
and $E^+=E_1^+\times E_2^+$.

\section{Some preliminaries}
In this section, we summarize some existing results in \cite{HeNi2016-1} for the following model:
\begin{equation}\label{lvs0}
\begin{cases}
  \ds\frac{\partial U}{\partial t}=d_1\Delta U+e^{-\gamma_1\tau_1}m_1(x)U-U^2-cUV, & x\in \Omega,\; t>0,\\
 \ds\frac{\partial V}{\partial t}=d_2\Delta V+e^{-\gamma_2\tau_2} m_2(x)V-bUV-V^2, & x\in\Omega,\; t>0,\\
\ds\f{\partial U}{\partial n}=\ds\f{\partial V}{\partial n}=0,& x\in\partial \Omega,\ t>0,\\
U(x,0)=U_0(x)\ge0,\;V(x,0)=V_0(x)\ge0,&x\in\Omega.
\end{cases}
\end{equation}
Clearly, under assumption $\mathbf{(M^+)}$, system \eqref{lvs0} has two semitrivial steady states $$\left(\theta_{d_1,\tau_1,\gamma_1,m_1},0\right)\;\; \text{and}\;\; \left(0,\theta_{d_2,\tau_2,\gamma_2,m_2}\right),$$
where $\theta_{d_i,\tau_i,\gamma_i,m_i}$ satisfies the  equation
\begin{equation}\label{theta}
\begin{cases}
 d_i\Delta \theta+e^{-\gamma_i\tau_i} m_i(x)\theta-\theta^2=0, & x\in \Omega,\\
\ds\f{\partial \theta}{\partial n}=0,& x\in\partial \Omega.
\end{cases}
\end{equation}
Denote by $\mu_1(d,w)$ the principal eigenvalue of the eigenvalue problem
\begin{equation}\label{seigen}
\begin{cases}
 d\Delta \phi+w(x)\phi=\mu\phi, & x\in \Omega,\\
\ds\f{\partial \phi}{\partial n}=0,& x\in\partial \Omega.\\
\end{cases}
\end{equation}
Then $\left(\theta_{d_1,\tau_1,\gamma_1,m_1},0\right)$ is linearly stable with respect to \eqref{lvs0} if
\begin{equation*}
\mu_1\left(d_2,e^{-\gamma_2\tau_2}m_2-b\theta_{d_1,\tau_1,\gamma_1,m_1}\right)<0,
\end{equation*}
is linearly unstable if
$$
\mu_1\left(d_2,e^{-\gamma_2\tau_2}m_2-b\theta_{d_1,\tau_1,\gamma_1,m_1}\right)>0,
$$
and is neutrally stable
if
$$
\mu_1\left(d_2,e^{-\gamma_2\tau_2}m_2-b\theta_{d_1,\tau_1,\gamma_1,m_1}\right)=0.
$$
Similarly, the linear stability of $\left(0,\theta_{d_2,\tau_2,\gamma_2,m_2}\right)$  with respect to \eqref{lvs0} is also determined by the sign of
$$\mu_1\left(d_1,e^{-\gamma_1\tau_1}m_1-c\theta_{d_2,\tau_2,\gamma_2,m_2}\right).$$
For fixed $b,c>0$, define
\begin{equation}\label{simpo}
\begin{split}
&S_u:=\{(d_1,d_2,\tau_1,\tau_2,\gamma_1,\gamma_2)\in\Gamma:\mu_1\left(d_2,e^{-\gamma_2\tau_2}m_2-b\theta_{d_1,\tau_1,\gamma_1,m_1}\right)<0\},\\
&S_v:=\{(d_1,d_2,\tau_1,\tau_2,\gamma_1,\gamma_2)\in\Gamma:\mu_1\left(d_1,e^{-\gamma_1\tau_1}m_1-c\theta_{d_2,\tau_2,\gamma_2,m_2}\right)<0\},\\
&S_{-}:=\{(d_1,d_2,\tau_1,\tau_2,\gamma_1,\gamma_2)\in\Gamma:\mu_1\left(d_2,e^{-\gamma_2\tau_2}m_2-b\theta_{d_1,\tau_1,\gamma_1,m_1}\right)>0,\\
&~~~~~~~~~~ \text{and}\;\mu_1\left(d_1,e^{-\gamma_1\tau_1}m_1-c\theta_{d_2,\tau_2,\gamma_2,m_2}\right)>0\},\\
&S_{u,0}:=\{(d_1,d_2,\tau_1,\tau_2,\gamma_1,\gamma_2)\in\Gamma:\mu_1\left(d_2,e^{-\gamma_2\tau_2}m_2-b\theta_{d_1,\tau_1,\gamma_1,m_1}\right)=0\},\\
&S_{v,0}:=\{(d_1,d_2,\tau_1,\tau_2,\gamma_1,\gamma_2)\in\Gamma:\mu_1\left(d_1,e^{-\gamma_1\tau_1}m_1-c\theta_{d_2,\tau_2,\gamma_2,m_2}\right)=0\},\\
&S_{0,0}:=\{(d_1,d_2,\tau_1,\tau_2,\gamma_1,\gamma_2)\in\Gamma:\mu_1\left(d_2,e^{-\gamma_2\tau_2}m_2-b\theta_{d_1,\tau_1,\gamma_1,m_1}\right)\\
&~~~~~~~~~~=\mu_1\left(d_1,e^{-\gamma_1\tau_1}m_1-c\theta_{d_2,\tau_2,\gamma_2,m_2}\right)=0\}.
\end{split}
\end{equation}
Then, we cite two main results in \cite{HeNi2016-1} as follows.
\begin{lemma}\label{posss}\cite[page 23]{HeNi2016-1}
Assume that $m_i(x)$ satisfies assumption $\mathbf{(M^+)}$ for $i=1,2$, and $0<bc\le1$. Then
for any $(d_1,d_2,\tau_1,\tau_2,\gamma_1,\gamma_2)\in \Gamma\setminus S_{0,0}$, every positive steady state of system \eqref{lvs0} is linearly stable if exists.
\end{lemma}

\begin{lemma}\cite[Theorem 1.3]{HeNi2016-1}\label{clarifi}
Assume that $m_i(x)$ satisfies assumption $\mathbf{(M^+)}$ for $i=1,2$, and $0<bc\le1$. Then
we have the following mutually disjoint decomposition of $\Gamma$:
\begin{equation*}
\Gamma=\left(S_u\cup S_{u,0}\setminus S_{0,0}\right)\cup\left(S_v\cup S_{v,0}\setminus S_{0,0}\right)\cup S_{-}\cup S_{0,0}.
\end{equation*}
Moreover, the following statements hold for model \eqref{lvs0}:
\begin{enumerate}
\item [(i)] For any $(d_1,d_2,\tau_1,\tau_2,\gamma_1,\gamma_2)\in (S_u\cup S_{u,0})\setminus S_{0,0}$, $\left(\theta_{d_1,\tau_1,\gamma_1,m_1},0\right)$ is globally asymptotically stable.
\item [(ii)] For any $(d_1,d_2,\tau_1,\tau_2,\gamma_1,\gamma_2)\in (S_v\cup S_{v,0})\setminus S_{0,0}$, $\left(0,\theta_{d_2,\tau_2,\gamma_2,m_2}\right)$ is globally asymptotically stable.
\item [(iii)] For any $(d_1,d_2,\tau_1,\tau_2,\gamma_1,\gamma_2)\in S_{-}$, model \eqref{lvs0} has a unique positive steady state, which is globally asymptotically stable.
\item [(iv)] For any $(d_1,d_2,\tau_1,\tau_2,\gamma_1,\gamma_2)\in S_{0,0}$, $\theta_{d_1,\tau_1,\gamma_1,m_1}\equiv c\theta_{d_2,\tau_2,\gamma_2,m_2}$, and model \eqref{lvs0} has a compact global attractor consisting of a continuum of steady states
    $$\{\left(\rho\theta_{d_1,\tau_1,\gamma_1,m_1},(1-\rho)\theta_{d_1,\tau_1,\gamma_1,m_1}/c\right):\rho\in(0,1)\}.$$
\end{enumerate}
\end{lemma}

\section{Global Dynamics}
In this section, we  give a complete classification of the global dynamics of model \eqref{lvs}, and our approach is motivated by the ones in \cite{HeNi2016-1}.
We first consider the eigenvalue problem associated with a positive steady state of \eqref{lvs}. Let
$(u,v)$ be a positive steady state of system \eqref{lvs}. Linearizing system \eqref{lvs} at $(u,v)$, we obtain the following eigenvalue problem
\begin{equation}\label{eig}
\begin{cases}
  \la\phi_1 =d_1\Delta \phi_1 +e^{-\gamma_1\tau_1-\la\tau_1}m_1(x)\phi_1 -(2u+cv)\phi_1-cu\phi_2, & x\in \Omega,\\
 \la\phi_2=d_2\Delta \phi_2+e^{-\gamma_2\tau_2-\la\tau_2} m_2(x)\phi_2-(bu+2v)\phi_2-bv\phi_1, & x\in\Omega,\\
\ds\f{\partial \phi_1}{\partial n}=\ds\f{\partial \phi_2}{\partial n}=0,& x\in\partial \Omega.
\end{cases}
\end{equation}
Then $(u,v)$ is linearly stable if all the eigenvalues of problem \eqref{eig} have negative real parts.
By virtue of the transformation $\psi_1=\phi_1$ and $\psi_2=-\phi_2$, the eigenvalue problem \eqref{eig} is equivalent to
\begin{equation}\label{eige}
\begin{cases}
  \la\psi_1 =d_1\Delta \psi_1 +e^{-\gamma_1\tau_1-\la\tau_1}m_1(x)\psi_1 -(2u+cv)\psi_1+cu\psi_2, & x\in \Omega,\\
 \la\psi_2=d_2\Delta \psi_2+e^{-\gamma_2\tau_2-\la\tau_2} m_2(x)\psi_2-(bu+2v)\psi_2+bv\psi_1, & x\in\Omega,\\
\ds\f{\partial \psi_1}{\partial n}=\ds\f{\partial \psi_2}{\partial n}=0,& x\in\partial \Omega.
\end{cases}
\end{equation}
Denote by $\la_1$ the principal eigenvalue of the following eigenvalue problem
\begin{equation}\label{eige0}
\begin{cases}
  \la\psi_1 =d_1\Delta \psi_1 +e^{-\gamma_1\tau_1}m_1(x)\psi_1 -(2u+cv)\psi_1+cu\psi_2, & x\in \Omega,\\
 \la\psi_2=d_2\Delta \psi_2+e^{-\gamma_2\tau_2} m_2(x)\psi_2-(bu+2v)\psi_2+bv\psi_1, & x\in\Omega,\\
\ds\f{\partial \psi_1}{\partial n}=\ds\f{\partial \psi_2}{\partial n}=0,& x\in\partial \Omega.
\end{cases}
\end{equation}
Then we show that the eigenvalue problem \eqref{eig} (or equivalently, \eqref{eige}) also has a principal eigenvalue $\tilde{\la}_1$, which has the same sign as
$\la_1$. We say that $\lambda$ is a principal eigenvalue of problem problem \eqref{eige0} (or respectively, \eqref{eige}) if
 \eqref{eige0} (or respectively, \eqref{eige}) has a solution $(\psi_1,\psi_2)>(0,0)$. Clearly,
 $$\la_1=\sup\{{\mathcal R}e \la:\la \;\;\text{is an eigenvalue of}\;\;\eqref{eige0}\},$$
and any eigenvalue $\lambda$ of \eqref{eige0} with $\lambda\ne\lambda_1$ satisfies ${\mathcal R}e \la<\lambda_1$.

The following result asserts the existence of principal eigenvalue for the eigenvalue problem \eqref{eige}, and the method that we use here for the proof is motivated by \cite{Thieme}.
\begin{theorem}\label{mai}
Assume that $m_i(x)$ satisfies assumption $\mathbf{(M^+)}$ for $i=1,2$, $d_1,d_2>0$, and $\tau_1,\tau_2,\gamma_1,\gamma_2\ge0$.
Then there exists a principal eigenvalue $\tilde\la_1$ of \eqref{eige}  with an associated eigenfunction $(\psi_1,\psi_2)>(0,0)$.
Furthermore,
\begin{enumerate}
\item [(i)] $\tilde\la_1=\sup\{{\mathcal R}e \la:\la \;\;\text{is an eigenvalue of}\;\;\eqref{eige}\}$,
\item [(ii)] $\tilde\la_1$ is simple and has the same sign as $\la_1$, where $\la_1$ is the principlal eigenvalue of \eqref{eige0},
\item [(iii)] any eigenvalue $\hat\lambda$ of \eqref{eige} with $\hat\lambda\ne\tilde \lambda_1$ satisfies ${\mathcal R}e \hat\la<\tilde \lambda_1$.
\end{enumerate}
\end{theorem}
\begin{proof}
If $\tau_1=\tau_2=0$, then the eigenvalue problem \eqref{eige} is reduced to \eqref{eige0}. Therefore, we only need to consider the case that at least one of $\tau_1$ and $\tau_2$ is positive.
Define $L=(L_1,L_2):E\to Y\times Y$ by
\begin{equation}\label{L}
\begin{split}
L_1(\psi_1,\psi_2)=&e^{-\gamma_1\tau_1}m_1(x)\psi_1(-\tau_1)+cu\psi_2(0),\\
L_2(\psi_1,\psi_2)=&e^{-\gamma_2\tau_2} m_2(x)\psi_2(-\tau_2)+bv\psi_1(0),\;\;(\psi_1,\psi_2)\in E,\\
\end{split}
\end{equation}
and $B=(B_1,B_2):\mathcal{D}(B)\subset Y\times Y\to Y\times Y$ by
\begin{equation}\label{B}
\begin{split}
B_1(\phi_1,\phi_2)=&d_1\Delta\phi_1-(2u+cv)\phi_1,\\
B_2(\phi_1,\phi_2)=&d_2\Delta \phi_2-(bu+2v)\phi_2, \;\;(\phi_1,\phi_2)\in \mathcal{D}(B).
\end{split}
\end{equation}
Clearly, the linear operator $L$ is positive, i.e., $L(E^+)\subset Y^+\times Y^+$. On the other hand, the linear operator $B$ generates a compact and analytic semigroup $T(t)$ on $Y\times Y$, and $T(t):Y\times Y\to Y\times Y$ is also positive. Let $U(t):E\to E$ be the solution semiflow associated with the abstract delayed linear equation
\begin{equation}\label{abs}
\begin{cases}
\ds\f{dV(t)}{dt}=B V(t)+LV_t, &t>0,\\
V(0)=\Psi_0=(\psi_{0,1},\psi_{0,2})\in E,
\end{cases}
\end{equation}
and let $A_U$ be its generator. Let $V(x,t,\Psi_0)=(v_1(x,t,\Psi_0),v_2(x,t,\Psi_0))$ be the solution of \eqref{abs} with initial value $\Psi_0=(\psi_{0,1},\psi_{0,2})\in E$. Then $U(t)\Psi_0=(v_1(x,t+\theta_1,\Psi_0),v_2(x,t+\theta_2,\Psi_0))\in E$, where $\theta_i\in[-\tau_i,0]$ for $i=1,2$.

We divide the following proof into several steps.

\noindent{\bf Step 1}. We  show that $U(t)$ is positive, i.e., $U(t)(E^+)\subset E^+$.

For convenience, we use $V(x,t)$ and $v_i(x,t)$ ($i=1,2$) to denote $V(x,t,\Psi_0)$ and $v_i(x,t,\Psi_0)$ ($i=1,2$).
Denote \begin{equation}\label{tildetau}
\tilde \tau:=\min\{\tau_1,\tau_2\}.
\end{equation} If $\tilde \tau>0$, then $V(x,t)$ satisfies that, for $t\in(0,\tilde\tau)$,
\begin{equation}\label{neg}
\begin{cases}
 \ds\f{\partial v_1}{\partial t} -d_1\Delta v_1+(2u+cv)v_1-cu v_2 =e^{-\gamma_1\tau_1}m_1(x)\psi_{0,1}(x,t-\tau_1)\ge0, &\;\;x\in \Omega,\\
\ds\f{\partial v_2}{\partial t}-d_2\Delta v_2+(bu+2v)v_2-bv v_1=e^{-\gamma_2\tau_2} m_2(x)\psi_{0,2}(x,t-\tau_2)\ge0, &\;\; x\in\Omega,\\
\ds\f{\partial v_1}{\partial n}=\ds\f{\partial v_2}{\partial n}=0, &\;\; x\in\partial \Omega.\\
\end{cases}
\end{equation}
It follows from the comparison principle that $v_i(x,t)\ge0$ for $(x,t)\in\overline\Omega\times [0,\tilde \tau]$ and $i=1,2$. By the method of step, we obtain that
\begin{equation}
v_i(x,t)\ge0\;\; \text{for}\;\; (x,t)\in\overline\Omega\times [0,\infty) \;\;\text{and}\;\;i=1,2,
\end{equation}
which implies that $U(t)(E^+)\subset E^+$.
If $\tilde\tau=0$, then $\tau_1=0$ or $\tau_2=0$. We only need to prove the case that $\tau_1>0$ and $\tau_2=0$, and the other case could be proved similarly. Then, for $t\in (0,\tau_1]$, $V(x,t)$ satisfies that
\begin{equation}\label{neg00}
\begin{cases}
 \ds\f{\partial v_1}{\partial t} -d_1\Delta v_1+(2u+cv)v_1-cu v_2 =e^{-\gamma_1\tau_1}m_1(x)\psi_{0,1}(x,t-\tau_1)\ge0, &\;\;x\in \Omega,\\
\ds\f{\partial v_2}{\partial t}-d_2\Delta v_2+[bu+2v-m_2(x)]v_2-bv v_1\ge0, &\;\; x\in\Omega,\\
\ds\f{\partial v_1}{\partial n}=\ds\f{\partial v_2}{\partial n}=0, &\;\; x\in\partial \Omega.
\end{cases}
\end{equation}
Similarly, we see from the comparison principle that $v_i(x,t)\ge0$ for $(x,t)\in\overline\Omega\times [0,\tau_1]$ and $i=1,2$. Then, by the method of step, we also obtain that  $U(t)(E^+)\subset E^+$ in this case.

\noindent{\bf Step 2}. Next we show that $U(t)$ is eventually strongly positive, i.e.,
there exists $t_*>0$ such that $U(t)(E^+\setminus \{\mathbf 0\})\subset int (E^+)$ for any $t>t_*$.
Here
\begin{equation*}
\begin{split}
&\{\mathbf 0\}=\{(\psi_1(x,\theta_1),\psi_2(x,\theta_2))\in E:\psi_i(x,\theta_i)\equiv 0\;\;\text{for}\;\;i=1,2\},\\
&int(E^+)=\{(\phi_1,\phi_2)\in E:\phi_i(x,\theta_i)>0 \;\;\text{for} \;\;x\in\overline \Omega,\;\theta_i\in[-\tau_i,0],\;i=1,2 \}.
\end{split}
\end{equation*}

Noticing that if
\begin{equation}\label{ini}\left(\psi_{0,1}(x,\theta_1),\;\psi_{0,2}(x,\theta_2)\right)\in E^+\setminus \{\mathbf 0\},\end{equation}
we have $\psi_{0,1}(x,\theta_1)\not\equiv 0$ or $\psi_{0,2}(x,\theta_2)\not\equiv 0$.
We only need to consider the case that $\psi_{0,1}(x,\theta_1)\not\equiv 0$, and the other case could be proved similarly.
If $\psi_{0,1}(x,\theta_1)\not\equiv 0$ and $\tau_1=0$, then it follows from the comparison principle that $v_1(x,t)>0$ for $x\in\overline\Omega$ and $t>\tau_1=0$.
If $\psi_{0,1}(x,\theta_1)\not\equiv 0$ and $\tau_1>0$, then
there exists $(x_0,\theta_0)\in \Omega\times (0,\tau_1)$ such that $\psi_{0,1}(x_0,-\theta_0)> 0$. We claim that $v_1(x,\tau_1-\theta_0)\not\equiv 0$. If it is not true, then
$v_1(x,\tau_1-\theta_0)\equiv 0$. This, combined with the first equation of \eqref{neg}, implies that
\begin{equation}\label{neg1}
 \ds\f{\partial v_1}{\partial t}(x_0,\tau_1-\theta_0) =e^{-\gamma_1\tau_1}m_1(x)\psi_{0,1}(x_0,-\theta_0) +cu v_2(x_0,\tau_1-\theta_0)>0.\\
\end{equation}
 Note that $v_1(x,t)\ge0$ for $(x,t)\in\overline\Omega\times [0,\infty)$ and $v_1(x_0,\tau_1-\theta_0)=0$.
 It follows that $\ds\f{\partial v_1}{\partial t}(x_0,\tau_1-\theta_0)=0$, which contradicts with \eqref{neg1}.
Then, $v_1(x,\tau_1-\theta_0)\not\equiv 0$. This, combined with the comparison principle, implies that
$v_1(x,t)>0$ for $x\in\overline\Omega$ and $t>\tau_1-\theta_0$.  Therefore, $v_1(x,t)>0$ for $x\in\overline\Omega$ and $t>\tau_1$.
Then, for $t>\tau_1$, $v_2(x,t)$ satisfies
\begin{equation*}
\begin{cases}
\ds\f{\partial v_2}{\partial t}-d_2\Delta v_2+(bu+2v)v_2\ge bv v_1>0, \;\; x\in\Omega,\\
\ds\f{\partial v_2}{\partial n}=0,\;\; x\in\partial \Omega.\\
\end{cases}
\end{equation*}
Similarly, we see from the comparison principle that $v_2(x,t)>0$ for $x\in\overline\Omega$ and $t>\tau_1$.
Note that  $U(t)\Psi_0=V_t(\Psi_0)$, where $V_t(\Psi_0)=(v_1(x,t+\theta_1),v_2(x,t+\theta_2))$ for $\theta_i\in[-\tau_i,0]$ and $i=1,2$.
It follows that $U(t)(E^+\setminus \{\mathbf 0\})\subset int (E^+)$ for any $t>2\tau_1+2\tau_2$.

\noindent{\bf Step 3}. Denote $s(A_U):=\sup\{{\mathcal{R}e \la:\la\in\sigma(A_U)}\}$. We prove that $s(A_U)$ is a simple eigenvalue of \eqref{eige} with a positive eigenfunction, and $s(A_U)$ has the same sign as the spectral bound $s(B+L_0)=\la_1$.

Since $U(t)$ is positive, it follows from \cite[Section 2]{Kerscher1984} that $s(A_U)$ is a spectral value of $A_U$. Define a operator $L_\la:Y\times Y\to Y\times Y$ by
$$L_\la\left(\phi_1, \phi_2\right)=L\left(
\phi_1e^{\la\theta_1}, \phi_2e^{\la\theta_2}\right),\;\;\theta_i\in[-\tau_i,0],\;i=1,2.$$
Then, from \cite[Section 4]{Kerscher1984}, we see that $s(A_U)$ has the same sign as the spectral bound $s(B+L_0)=\la_1$.
From \cite[Chapter 3]{wu1996theory}, we see that $\la\in\sigma_p(A_U)$ if and only if $\la$ is an eigenvalue of problem \eqref{eige}, and the corresponding eigenfunction of $A_U$ with respect to $\la$ is $(\psi_1 e^{\la\theta_1},\psi_2 e^{\la\theta_2})$ where $\theta_i\in[-\tau_i,0]$ for $i=1,2$ and $(\psi_1, \psi_2)$ is the corresponding eigenfunction of \eqref{eige} with respect to $\la$.
Therefore, we only need to show that
$s(A_U)\in \sigma_p(A_U)$ and the associated eigenfunction $(\psi^s_1 e^{s(A_U)\theta_1},\psi^s_2 e^{s(A_U)\theta_2})$  where $\theta_i\in[-\tau_i,0]$, $i=1,2$, is strongly positive, i.e., $(\psi^s_1,\psi^s_2)>(0,0)$.

It follows from \cite[Chapter 3]{wu1996theory} that $U(t):E\to E$ is compact for $t>\tau_1+\tau_2$.
Note that, for a fixed $t_0>2\tau_1+2\tau_2$, $U(t_0)$ is strongly positive. Then we see from the Krein-Rutman theorem (see \cite[Theorem 3.2]{Amann}) that the spectral radius $r(U(t_0))$ is positive and a simple eigenvalue eigenvalue
of $U(t_0)$ associated with an eigenfunction in $int(E^+)$, and any eigenvalue $\mu$ of $U(t_0)$ with $\mu\ne r(U(t_0))$ satisfies $|\mu|< r(U(t_0))$.
Then from \cite[Theorem 2.2.4]{pazy}, we obtain that there exists $\tilde\la\in\sigma_p(A_U)$ such that
$r(U(t_0))=e^{\tilde \la t_0}$. We claim that $\tilde \la\in\mathbb{R}$.
If it is not true, then $\tilde\la\in\mathbb{C}\setminus\mathbb{R}$.
Note that
\begin{equation}
\begin{split}
U(t_0)(\tilde \psi_1e^{\tilde \la\theta_1}, \tilde\psi_2e^{\tilde \theta_2})=e^{\tilde\la t_0}(\tilde \psi_1e^{\tilde\la\theta_1}, \tilde \psi_2e^{\tilde\la\theta_2})
=r(U(t_0))(\tilde \psi_1e^{\tilde \la\theta_1}, \tilde\psi_2e^{\tilde\la\theta_2}),
\end{split}
\end{equation}
where $(\tilde\psi_1(x),\tilde\psi_2(x))$ is the corresponding eigenfunction with respect to $\tilde\la$ for \eqref{eige}. Then
we have $(\tilde \psi_1e^{\tilde \la\theta_1}, \tilde\psi_2e^{\tilde\la\theta_2})\in int(E^+)$. If $\tau_1=\tau_2=0$, then $\tilde\psi_1,\tilde\psi_2>0$, which yields $\tilde\la\in\mathbb{R}$.
This is a contradiction. If $\tau_i\ne0$, then $\tilde \psi_i e^{\tilde \la\theta_i}\not \in int(E_i^+)$ for $i=1,2$, which is also a contradiction.
Therefore the claim is true, and consequently, $\tilde \la\le s(A_U)$.

Noticing that $s(A_U)$ is a spectral value of $A_U$, we see from \cite[Theorem 2.2.3]{pazy} that $e^{s(A_U)t_0}\in \sigma(U(t_0))$, which implies that $e^{s(A_U)t_0}\le r(U(t_0))=e^{\tilde \la t_0}$. Therefore, $s(A_U)= \tilde \la\in\sigma_p(A_U)$, and consequently, $s(A_U)$ is an eigenvalue of problem \eqref{eige} with the corresponding eigenfunction $(\psi^s_1, \psi^s_2)$.
Since
\begin{equation}
\begin{split}
U(t_0)\left(\psi^s_1e^{s(A_U)\theta_1}, \psi^s_2e^{s(A_U)\theta_2}\right)=&e^{s(A_U)t_0}\left(\psi^s_1e^{s(A_U)\theta_1}, \psi^s_2e^{s(A_U)\theta_2}\right)\\
=&r(U(t_0))\left(\psi^s_1e^{s(A_U)\theta_1}, \psi^s_2e^{s(A_U)\theta_2}\right),
\end{split}
\end{equation}
it follows from the Krein-Rutman theorem that $s(A_U)$ is simple and $\psi^s_i>0$ for $i=1,2$.
From the above three steps, we see that $\tilde\la_1=s(A_U)$ is the principle eigenvalue of \eqref{eige}, and $(i)$ and $(ii)$ hold.

\noindent{\bf Step 4}. We prove that $(iii)$ holds.

We firstly claim that, for any $\hat\la\in\sigma_p(A_U)$ and $\hat\la\ne s(A_U)$, $e^{\hat\la t_0}\ne e^{s(A_U)t_0}$. If it is not true, then $e^{\hat\la t_0}=e^{s(A_U)t_0}$, and consequently, there exist an integer $k\ne0$ and a constant $c_0(\ne 0)\in\mathbb{C}$ such that
$$\hat\la=s(A_U)+\ds\f{2k\pi }{t_0}{\rm i},$$
and $$\left(\hat\psi_1 e^{\hat\la\theta_1},\hat\psi_2 e^{\hat\la\theta_2}\right)=c_0\left(\psi^s_1e^{s(A_U)\theta_1}, \psi^s_2e^{s(A_U)\theta_2}\right)
\;\;\text{for any}\;\;x\in\overline\Omega,\;\theta_i\in[-\tau_i,0],\;i=1,2,$$
 where $(\hat\psi_1(x),\hat\psi_2(x))$ is the corresponding eigenfunction with respect to $\hat\la$ for \eqref{eige}.
This is a contradiction. Therefore, the claim is true, and from
the Krein-Rutman theorem, we have
\begin{equation*}
|e^{\hat\la t_0}|=e^{t_0{\mathcal R}e \hat \la }<r\left(U(t_0)\right)=e^{s(A_U)t_0},
\end{equation*}
which implies that ${\mathcal R}e \hat \la<s(A_U)$.
\end{proof}

Next we consider the eigenvalue problems associated with \eqref{lvs} with respect to the semitrivial steady states $(\theta_{d_1,\tau_1,\gamma_1,m_1},0)$ and $(0,\theta_{d_2,\tau_2,\gamma_2,m_2})$. Linearizing system \eqref{lvs} at $(\theta_{d_1,\tau_1,\gamma_1,m_1},0)$, we  obtain the following eigenvalue problem
\begin{equation}\label{eigsemi}
\begin{cases}
  \la\phi_1 =d_1\Delta \phi_1 +e^{-\gamma_1\tau_1-\la\tau_1}m_1(x)\phi_1 -2\theta_{d_1,\tau_1,\gamma_1,m_1}\phi_1-c\theta_{d_1,\tau_1,\gamma_1,m_1}\phi_2, & x\in \Omega,\\
 \la\phi_2=d_2\Delta \phi_2+e^{-\gamma_2\tau_2-\la\tau_2} m_2(x)\phi_2-b\theta_{d_1,\tau_1,\gamma_1,m_1}\phi_2, & x\in\Omega,\\
\ds\f{\partial \phi_1}{\partial n}=\ds\f{\partial \phi_2}{\partial n}=0,& x\in\partial \Omega.
\end{cases}
\end{equation}
Therefore, we only need to consider the following eigenvalue problem
\begin{equation}\label{eigsemisemi}
\begin{cases}
 \la\phi_2=d_2\Delta \phi_2+e^{-\gamma_2\tau_2-\la\tau_2} m_2(x)\phi_2-b\theta_{d_1,\tau_1,\gamma_1,m_1}\phi_2, & x\in\Omega,\\
\ds\f{\partial \phi_2}{\partial n}=0,& x\in\partial \Omega.\\
\end{cases}
\end{equation}
Similarly, the eigenvalue problem with respect to $(0,\theta_{d_2,\tau_2,\gamma_2,m_2})$ takes the following form:
\begin{equation}\label{eigsemi2}
\begin{cases}
  \la\phi_1 =d_1\Delta \phi_1 +e^{-\gamma_1\tau_1-\la\tau_1}m_1(x)\phi_1 -c\theta_{d_2,\tau_2,\gamma_2,m_2}\phi_1, & x\in \Omega,\\
 \la\phi_2=d_2\Delta \phi_2+e^{-\gamma_2\tau_2-\la\tau_2} m_2(x)\phi_2-2\theta_{d_2,\tau_2,\gamma_2,m_2}\phi_2-b\theta_{d_2,\tau_2,\gamma_2,m_2}\phi_1, & x\in\Omega,\\
\ds\f{\partial \phi_1}{\partial n}=\ds\f{\partial \phi_2}{\partial n}=0,& x\in\partial \Omega,\\
\end{cases}
\end{equation}
and we also only need to consider the following eigenvalue problem
\begin{equation}\label{eigsemisemi2}
\begin{cases}
\la\phi_1 =d_1\Delta \phi_1 +e^{-\gamma_1\tau_1-\la\tau_1}m_1(x)\phi_1 -c\theta_{d_2,\tau_2,\gamma_2,m_2}\phi_1, & x\in \Omega,\\
\ds\f{\partial \phi_1}{\partial n}=0,& x\in\partial \Omega.\\
\end{cases}
\end{equation}
By virtue of the similar arguments as Theorem \ref{mai} (see also \cite[Lemma 2.2]{GuoY2018}), we have the following results on the principal eigenvalues of \eqref{eigsemisemi} and \eqref{eigsemisemi2}.
\begin{theorem}\label{maisemi}
Assume that $m_i(x)$ satisfies assumption $\mathbf{(M^+)}$ for $i=1,2$, $d_1,d_2>0$, and $\tau_1,\tau_2,\gamma_1,\gamma_2\ge0$.
Then
\begin{enumerate}
\item [$(i)$]  problem \eqref{eigsemisemi} has a principal eigenvalue $\tilde \mu_1$, where $$\tilde\mu_1=\sup\{{\mathcal R}e \la:\la \;\;\text{is an eigenvalue of}\;\;\eqref{eigsemisemi}\},$$ and $\tilde \mu_1$ has the same sign as
$$\mu_1\left(d_2,e^{-\gamma_2\tau_2}m_2-b\theta_{d_1,\tau_1,\gamma_1,m_1}\right).$$
\item [$(ii)$] problem \eqref{eigsemisemi2} has a principal eigenvalue $\hat\mu_1$, where $$\hat\mu_1=\sup\{{\mathcal R}e \la:\la \;\;\text{is an eigenvalue of}\;\;\eqref{eigsemisemi2}\},$$  and $\hat \mu_1$ has the same sign as $$\mu_1\left(d_1,e^{-\gamma_1\tau_1}m_1-c\theta_{d_2,\tau_2,\gamma_2,m_2}\right).$$
\end{enumerate}
\end{theorem}

Finally we show that system \eqref{lvs} generates a monotone dynamical system.
\begin{proposition}\label{monot}
Let $(U_i(x,t),V_i(x,t))$ be the corresponding solution of model \eqref{lvs} with initial value $(U_{0,i},V_{0,i})$ for $i=1,2$. Assume that
\begin{equation*}
\begin{split}
&U_{0,1}\ge U_{0,2}\ge0 \;\;\text{for}\;\;x\in\overline\Omega,\;t\in[-\tau_1,0],\\
&0\le V_{0,1}\le V_{0,2}\;\;\text{for}\;\;x\in\overline\Omega,\;t\in[-\tau_2,0].
 \end{split}
 \end{equation*}
Then  $$U_1(x,t)\ge U_2(x,t)\;\;\text{and}\;\; V_1(x,t)\le V_2(x,t)\;\;\text{for}\;\;x\in\overline\Omega,\;t\ge0.$$
\end{proposition}
\begin{proof}
We only prove the case that $\tau_1,\tau_2\ne0$, and other cases could be proved similarly.
Let $\overline U(x,t)=U_1(x,t)-U_2(x,t)$, $\overline V(x,t)=V_2(x,t)-V_1(x,t)$, $\overline U_{0}=U_{0,1}-U_{0,2}$ and $\overline V_{0}=V_{0,2}-V_{0,1}$ for $i=1,2$.
Then $(\overline U(x,t),\overline V(x,t))$ satisfies
\begin{equation*}
\begin{cases}
  \ds\frac{\partial \overline U}{\partial t}-d_1\Delta \overline U+(U_1+U_2+cV_1)\overline U-cU_2\overline V\ge0, & x\in \Omega,\; t\in[0,\tilde \tau],\\
 \ds\frac{\partial \overline V}{\partial t}-d_2\Delta\overline  V+(V_1+V_2+b U_2)\overline V-b V_1\overline U\ge0, & x\in\Omega,\; t\in[0,\tilde \tau],\\
\ds\f{\partial \overline U}{\partial n}=\ds\f{\partial \overline V}{\partial n}=0,& x\in\partial \Omega,\;t\in[0,\tilde \tau],\\
\overline U(x,0)=\overline U_0(x,0)\ge0,\;\overline V(x,0)=\overline V_0(x,0)\ge0,&x\in\Omega,
\end{cases}
\end{equation*}
where $\tilde\tau$ is defined as in Eq. \eqref{tildetau}.
It follows from the comparison principle that $\overline U(x,t),\overline V(x,t)\ge0$ for $x\in\overline\Omega$ and $t\in[0,\tilde \tau]$.
Then, by the method of step, we could prove that $\overline U(x,t),\overline V(x,t)\ge0$ for $x\in\overline\Omega$ and $t\ge0$. This completes the proof.
\end{proof}

Then from Lemmas \ref{posss}, \ref{clarifi}, Theorems \ref{mai}, \ref{maisemi} and Proposition \ref{monot}, we can obtain the following complete classification on the global dynamics of model \eqref{lvs} for $0<bc\le1$.
\begin{theorem}\label{clarifi2}
Assume that $m_i(x) $ satisfies assumption $\mathbf{(M^+)}$ for $i=1,2$, and $0<bc\le1$. Then
we have the following mutually disjoint decomposition of $\Gamma$:
\begin{equation*}
\Gamma=\left(S_u\cup S_{u,0}\setminus S_{0,0}\right)\cup\left(S_v\cup S_{v,0}\setminus S_{0,0}\right)\cup S_{-}\cup S_{0,0}.
\end{equation*}
Moreover, the following statements hold for model \eqref{lvs}:
\begin{enumerate}
\item [(i)] For any $(d_1,d_2,\tau_1,\tau_2,\gamma_1,\gamma_2)\in \left(S_u\cup S_{u,0}\setminus S_{0,0}\right)$, $\left(\theta_{d_1,\tau_1,\gamma_1,m_1},0\right)$ is globally asymptotically stable.
\item [(ii)] For any $(d_1,d_2,\tau_1,\tau_2,\gamma_1,\gamma_2)\in \left(S_v\cup S_{v,0}\setminus S_{0,0}\right)$, $\left(0,\theta_{d_2,\tau_2,\gamma_2,m_2}\right)$ is globally asymptotically stable.
\item [(iii)] For any $(d_1,d_2,\tau_1,\tau_2,\gamma_1,\gamma_2)\in S_{-}$, model \eqref{lvs} has a unique positive steady state, which is globally asymptotically stable.
\item [(iv)] For any $(d_1,d_2,\tau_1,\tau_2,\gamma_1,\gamma_2)\in S_{0,0}$, $\theta_{d_1,\tau_1,\gamma_1,m_1}\equiv c\theta_{d_2,\tau_2,\gamma_2,m_2}$, and model \eqref{lvs} has a compact global attractor consisting of a continuum of steady states
    $$\{\left(\rho\theta_{d_1,\tau_1,\gamma_1,m_1},(1-\rho)\theta_{d_1,\tau_1,\gamma_1,m_1}/c\right):\rho\in(0,1)\}.$$
\end{enumerate}
\end{theorem}
\begin{proof}
We only prove $(iii)$, and other cases could be proved similarly. We see from Theorem \ref{monot} that
system \eqref{lvs} generates a monotone dynamical system. It follows from Lemma \ref{posss} and Theorem \ref{mai}
that, for any $(d_1,d_2,\tau_1,\tau_2,\gamma_1,\gamma_2)\in \Gamma\setminus S_{0,0}$, every positive steady state of system \eqref{lvs} is linearly stable if exists. Note that, for $(d_1,d_2,\tau_1,\tau_2,\gamma_1,\gamma_2)\in S_{-}$, each of the two semi-trivial steady state is unstable. Then we see from the theory of monotone dynamical system (\cite[Proposition 9.1 and Theorem 9.2]{Hess}) that system \eqref{lvs} has a unique positive steady state, which is globally asymptotically stable.
\end{proof}
\begin{remark}
We see from the proof of \cite[Theorem 1.3]{HeNi2016-1} that if $S_{0,0}\ne \emptyset $, then $bc=1$. Therefore, if $0<b,c<1$ (the weak competition case), $S_{0,0}=\emptyset$. Then, for the weak competition case, the dynamics of model \eqref{lvs} can be classified as follows:
 either one of the two semitrivial steady states is globally asymptotically stable, or there exists a unique positive steady state which is globally asymptotically stable.
\end{remark}

\section{Applications and Discussion}
In this section, we first apply the obtained results in Section 3 to two concrete examples and show the effect of delays. Then we give some discussion and show that the method for model \eqref{lvs} can also be applied to another delayed competition model.

\subsection{Example (A)}
Firstly, we consider a special case and show the effect of delays $\tau_1$ and $\tau_2$.  By using the approach of adaptive dynamics \cite{Dieckmann}, we assume that $d_1=d_2=d$, $\gamma_1=\gamma_2=\gamma$, $b=c=1$, $m_1(x)=m_2(x)=m(x)$, and $\tau_1\ne\tau_2$. That is,
the two species are supposed to be identical except their maturation times:
\begin{equation}\label{lvsident}
\begin{cases}
  \ds\frac{\partial U}{\partial t}=d\Delta U+e^{-\gamma\tau_1}m(x)U(x,t-\tau_1)-U^2-UV, & x\in \Omega,\; t>0,\\
 \ds\frac{\partial V}{\partial t}=d\Delta V+e^{-\gamma\tau_2} m(x)V(x,t-\tau_2)-UV-V^2, & x\in\Omega,\; t>0,\\
\ds\f{\partial U}{\partial n}=\ds\f{\partial V}{\partial n}=0,& x\in\partial \Omega,\; t>0,\\
U(x,t)=U_0(x,t)\ge0, &x\in\Omega,\;t\in [-\tau_1,0],\\
V(x,t)=V_0(x,t)\ge0, &x\in\Omega,\;t\in [-\tau_2,0].\\
\end{cases}
\end{equation}
For simplicity of notations, we use $\theta_{\tau_1}$ and $\theta_{\tau_2}$ to denote $\theta_{d_1,\tau_1,\gamma_1,m_1}$ and $\theta_{d_2,\tau_2,\gamma_2,m_2}$, respectively.
Then the global dynamics of model \eqref{lvsident} can be classified as the following.
\begin{theorem}\label{ident}
Assume that $m(x)\in C^{\alpha}(\overline \Omega)$ $(\alpha\in(0,1))$, $m(x)>0$ on $\overline\Omega$, and $d,\gamma,\tau_1,\tau_2>0$. Then the following three statements hold.
\begin{enumerate}
\item [$(i)$] If $\tau_1>\tau_2$, then $(0,\theta_{\tau_2})$ is globally asymptotically stable.
\item [$(ii)$] If $\tau_1<\tau_2$, then $(\theta_{\tau_1},0)$ is globally asymptotically stable.
\item [$(iii)$] If $\tau_1=\tau_2$, then model \eqref{lvsident} has a compact global attractor consisting of a continuum of steady states
    $\{\left(\rho\theta_{\tau_1},(1-\rho)\theta_{\tau_1}\right):\rho\in(0,1)\}.$
\end{enumerate}
\end{theorem}
\begin{proof}
If $\tau_1>\tau_2$, then $\theta_{\tau_1}$ satisfies
\begin{equation}
\begin{split}
-d\Delta \theta_{\tau_1}=&\theta_{\tau_1}\left[e^{-\gamma\tau_1}m(x)-\theta_{\tau_1}\right]\\
<&\theta_{\tau_1}\left[e^{-\gamma\tau_2}m(x)-\theta_{\tau_1}\right],
\end{split}
\end{equation}
which implies that $\theta_{\tau_1}<\theta_{\tau_2}$ from the comparison principle. Noticing that
\begin{equation}
\begin{split}
-d\Delta \theta_{\tau_1}=&\theta_{\tau_1}\left[e^{-\gamma\tau_1}m(x)-\theta_{\tau_1}\right],\\
-d\Delta \theta_{\tau_2}=&\theta_{\tau_2}\left[e^{-\gamma\tau_2}m(x)-\theta_{\tau_2}\right],
\end{split}
\end{equation}
we have
\begin{equation}\label{0equa}
\mu_1\left(d,e^{-\gamma\tau_1}m(x)-\theta_{\tau_1}\right)=0,\;\;\mu_1\left(d,e^{-\gamma\tau_2}m(x)-\theta_{\tau_2}\right)=0.
\end{equation}
Therefore, for $\tau_1>\tau_2$,
\begin{equation}
\mu_1\left(d,e^{-\gamma\tau_1}m(x)-\theta_{\tau_2}\right)<0,\;\;\mu_1\left(d,e^{-\gamma\tau_2}m(x)-\theta_{\tau_1}\right)>0,
\end{equation}
which implies that $(0,\theta_{\tau_2})$ is globally asymptotically stable from Theorem \ref{clarifi2}. Similarly, we can prove part $(ii)$.
Part $(iii)$ could be obtained directly  from Eq. \eqref{0equa} and Theorem \ref{clarifi2}.
\end{proof}
Theorem \ref{ident} implies that the species with shorter maturation time will prevail if all other conditions (dispersal, growth) are identical, see Fig. \ref{figd} for the diagram of the global dynamics of model \eqref{lvsident} and Fig. \ref{fig2} for the numerical simulations.
\begin{figure}[htbp]
\centering\includegraphics[width=0.5\textwidth]{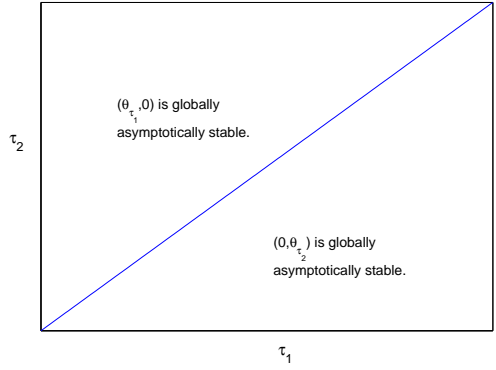}
\caption{The dynamics of model \eqref{lvsident}. \label{figd}}
\end{figure}

 \begin{figure}[htbp]
\centering\includegraphics[width=0.5\textwidth]{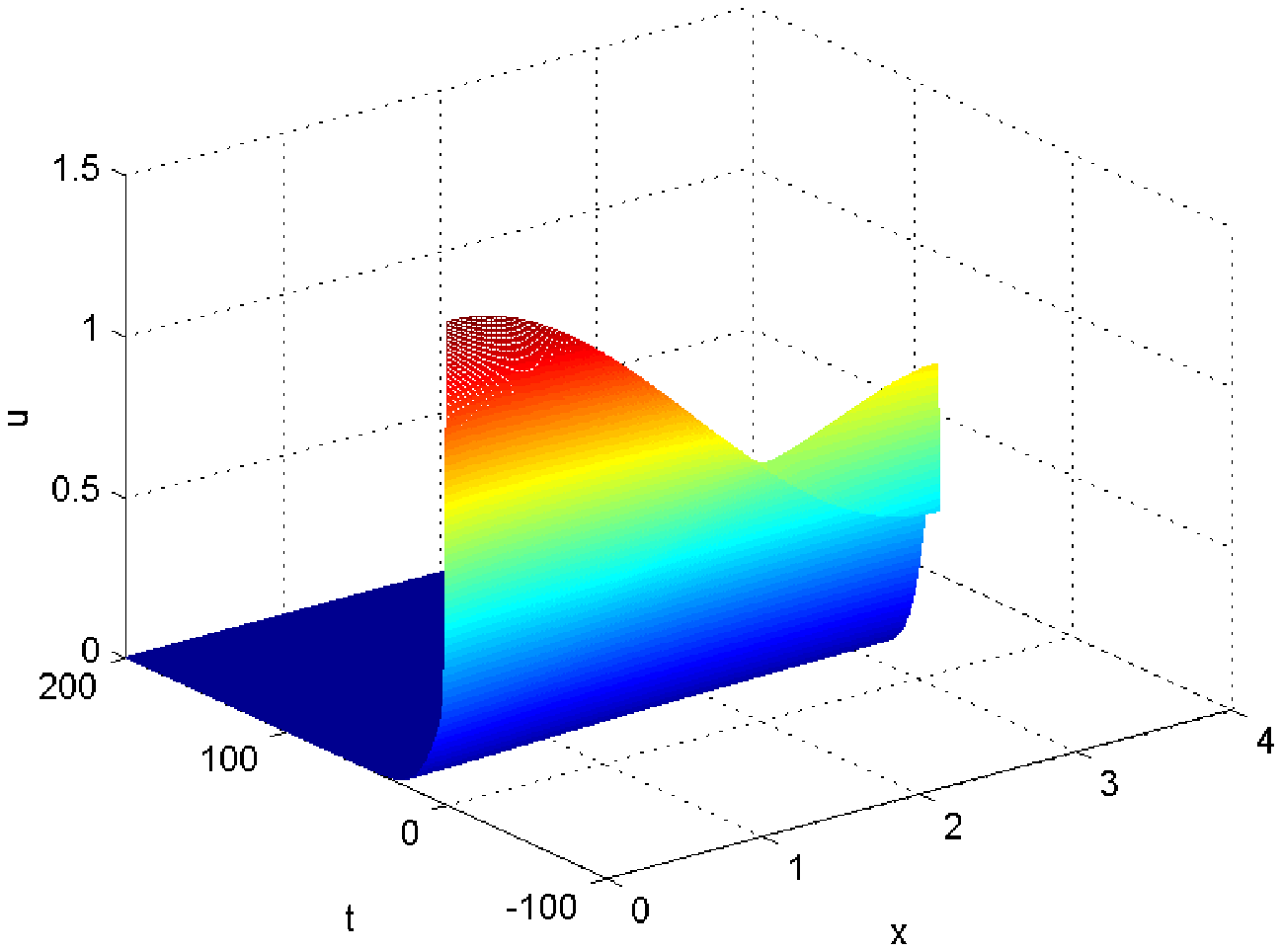}\includegraphics[width=0.5\textwidth]{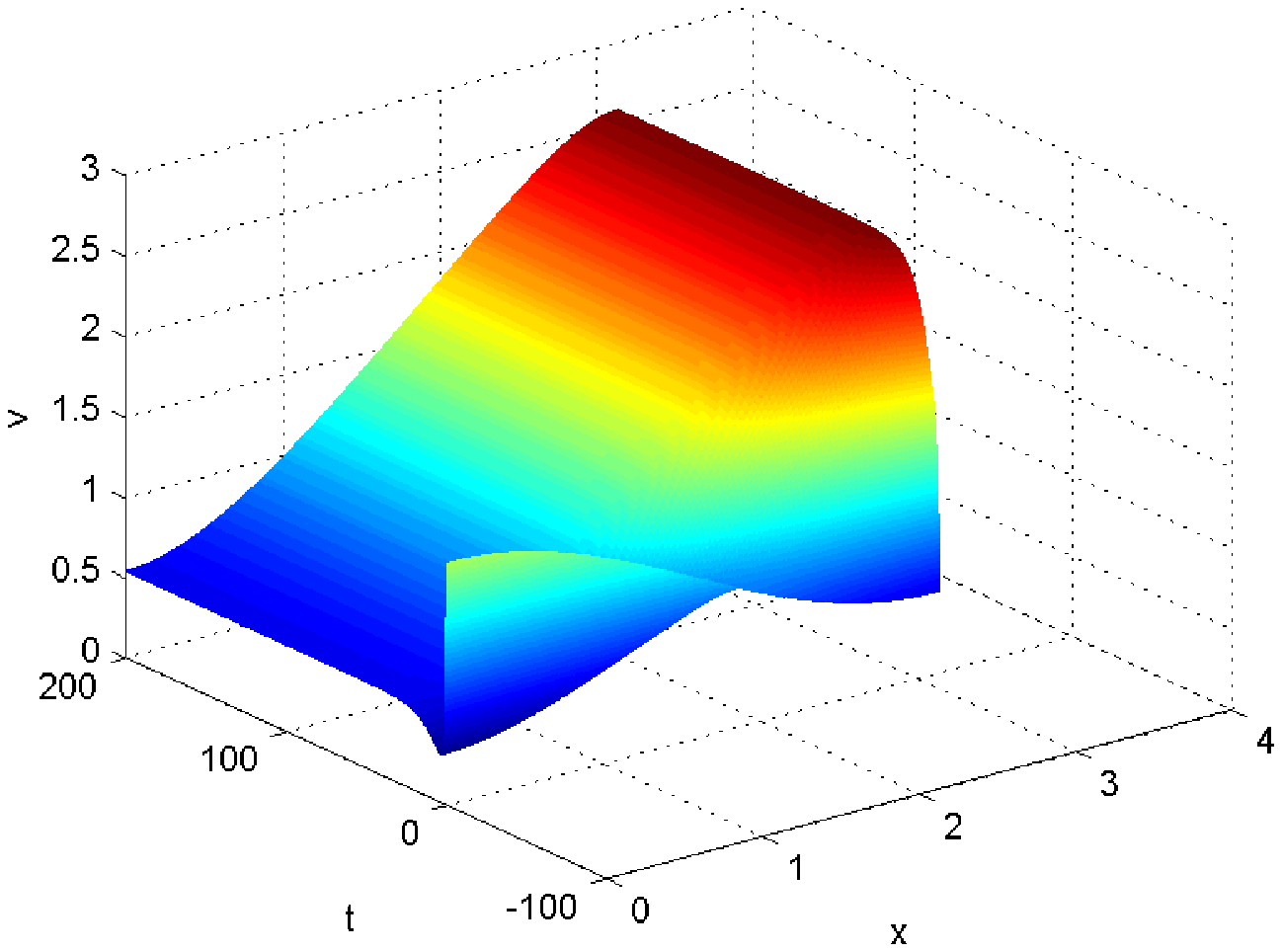}
\caption{The solution of model \eqref{lvsident} converges to the semitrivial steady state $(0,\theta_{\tau_2})$ for $\tau_1>\tau_2$. Here $d=0.2$, $\gamma=1$, $\tau_1=0.2$, $\tau_2=0.1$, $m(x)=x$, $\Omega=(0,\pi)$, and the initial values $u(x,t)=1+0.5\cos x$ for $x\in\overline\Omega,t\in[-\tau_1,0]$, and $v(x,t)=1+0.5\cos x$ for $x\in\overline\Omega,t\in[-\tau_2,0]$.\label{fig2}}
\end{figure}

\subsection{Example (B)}
In this subsection, we assume that $\gamma_2=\tau_2=0$, and revisit the model investigated in \cite{GuoY2018}. That is,
\begin{equation}\label{lvs00}
\begin{cases}
  \ds\frac{\partial U}{\partial t}=d_1\Delta U+e^{-\gamma_1\tau_1}m_1(x)U(x,t-\tau_1)-U^2-cUV, & x\in \Omega,\; t>0,\\
 \ds\frac{\partial V}{\partial t}=d_2\Delta V+m_2(x)V-bUV-V^2, & x\in\Omega,\; t>0,\\
\ds\f{\partial U}{\partial n}=\ds\f{\partial V}{\partial n}=0,& x\in\partial \Omega,\ t>0,\\
U(x,t)=U_0(x,t)\ge0, &x\in\Omega,\;t\in [-\tau_1,0],\\
V(x,t)=V_0(x,t)\ge0,&x\in\Omega,\;t=0.\\
\end{cases}
\end{equation}
We also consider the effect of delay for model \eqref{lvs00}, and the method is motivated by \cite{HeNi2013-2}. If $m_i(x)$ satisfies assumption $\mathbf{(M^+)}$ for $i=1,2$, then system \eqref{lvs00} has two semitrivial steady states
$$\left(\theta_{d_1,\tau_1,\gamma_1,m_1},0 \right)\;\;\text{and}\;\;\left(0,\theta_{d_2,0,0,m_2}\right).$$ Denote
\begin{equation}\label{simpo00}
\begin{split}
&\tilde S_u:=\{(d_1,d_2):(d_1,d_2,0,0,0,0)\in S_p\} \;\;\text{for}\;\; p=u,v,-,\\
&\tilde S_{p,0}:=\{(d_1,d_2):(d_1,d_2,0,0,0,0)\in S_{p,0}\} \;\;\text{for}\;\; p=u,v,0,\\
\end{split}
\end{equation}
where $S_u$, $S_v$, $S_-$, $S_{u,0}$, $S_{v,0}$ and $S_{0,0}$ are defined as in Eq. \eqref{simpo}.
It follows from \cite[Theorem 1.3]{HeNi2016-1} that if $m_i(x)$ satisfies assumption $\mathbf{(M^+)}$ for $i=1,2$, and $0<bc\le1$, then $\left(\mathbb{R}^+\right)^2$ has the following mutually disjoint decomposition:
\begin{equation}\label{td1d2}
\left(\mathbb{R}^+\right)^2=(\tilde S_u\cup\tilde S_{u,0}\setminus \tilde S_{0,0})\cup (\tilde S_v\cup\tilde S_{v,0}\setminus \tilde S_{0,0})\cup \tilde S_{-} \cup \tilde S_{0,0}.
\end{equation}\label{ex2}
Then we have the following results.
\begin{theorem}
Assume that $m_i(x)$ satisfies assumption $\mathbf{(M^+)}$ for $i=1,2$, and $0<bc\le1$. The following statements hold for system \eqref{lvs00}.
\begin{enumerate}
\item [$(i)$] If $(d_1,d_2)\in (\tilde S_v\cup\tilde S_{v,0}\setminus \tilde S_{0,0})\cup \tilde S_{0,0}=\tilde S_v\cup\tilde S_{v,0}$, then the semitrivial steady state
$\left(0,\theta_{d_2,0,0,m_2}\right)$ is globally asymptotically stable for any $\gamma_1,\tau_1>0$.
\item [$(ii)$] If $(d_1,d_2)\in \tilde S_{-}\cup\left(\tilde S_{u,0}\setminus \tilde S_{0,0}\right)$, then there exists $\tilde \delta\in(0,1)$ such that the semitrivial steady state
$\left(0,\theta_{d_2,0,0,m_2}\right)$ is globally asymptotically stable for $\gamma_1\tau_1\ge-\ln \tilde\delta$, and for $0<\gamma_1\tau_1<-\ln \tilde\delta$, system \eqref{lvs00} has a unique positive steady state, which is globally asymptotically stable.
\item [$(iii)$] If $(d_1,d_2)\in \tilde S_u$, then there exist $0<\delta_1\le\delta_2<1$ such that
\begin{equation*}
\begin{split}
&\mu_1\left(d_1,e^{-\gamma_1\tau_1}m_1-c\theta_{d_2,0,0,m_2}\right)=0\;\;\text{for}\;\;\gamma_1\tau_1=-\ln\delta_1,\\
&\mu_1\left(d_2,m_2-b\theta_{d_1,\tau_1,\gamma_1,m_1}\right)=0 \;\;\text{for}\;\; \gamma_1\tau_1=-\ln\delta_2.
\end{split}
\end{equation*}
Moreover,
\begin{enumerate}
\item [$(iii_1)$] if $\delta_1<\delta_2$, then $\left(\theta_{d_1,\tau_1,\gamma_1,m_1},0\right)$ is globally asymptotically stable for $0<\gamma_1\tau_1\le -\ln\delta_2$, $\left(0,\theta_{d_2,0,0,m_2}\right)$ is globally asymptotically stable for $\gamma_1\tau_1\ge-\ln \delta_1$, and for $-\ln\delta_2<\gamma_1\tau_1<-\ln\delta_1$, system \eqref{lvs00} has a unique positive steady state, which is globally asymptotically stable;
\item [$(iii_2)$] if $\delta_1=\delta_2$, then $\left(\theta_{d_1,\tau_1,\gamma_1,m_1},0\right)$ is globally asymptotically stable for $0<\gamma_1\tau_1< -\ln\delta_1$, $\left(0,\theta_{d_2,0,0,m_2}\right)$ is globally asymptotically stable for $\gamma_1\tau_1>-\ln \delta_1$, and for $\gamma_1\tau_1=-\ln\delta_1$, system \eqref{lvs00} has a compact global attractor consisting of a continuum of steady states.
\end{enumerate}
\end{enumerate}
\end{theorem}
\begin{proof}
Denote $\delta=e^{-\gamma_1\tau_1}$, $\theta_{1,\delta}=\theta_{d_1,\tau_1,\gamma_1,m_1}$ and $\theta_2=\theta_{d_2,0,0,m_2}$.
Then
$\delta\in(0,1)$, $\theta_{1,\delta}$ depending on $\delta$ satisfies
\begin{equation*}
\begin{cases}
d_1 \Delta u+u(\delta m_1(x)-u)=0, &x\in\Omega,\\
\ds\f{\partial u}{\partial n}=0,&x\in\partial \Omega,
\end{cases}
\end{equation*}and $\theta_2$ satisfies
\begin{equation*}
\begin{cases}
d_2 \Delta v+u(m_2(x)-v)=0, &x\in\Omega,\\
\ds\f{\partial v}{\partial n}=0,&x\in\partial \Omega.
\end{cases}
\end{equation*}
Let $\theta_{1,\delta}=\delta\tilde \theta_{1,\delta}$, and a direct computation implies that $\tilde \theta_{1,\delta}$ satisfies
\begin{equation}
\begin{cases}
d_1 \Delta u+\delta u( m_1(x)-u)=0, &x\in\Omega,\\
\ds\f{\partial u}{\partial n}=0,&x\in\partial \Omega.
\end{cases}
\end{equation}
It follows from \cite[Theorem 1.1]{Lou} that
$$\lim_{\delta\to0} \tilde\theta_{1,\delta}=\ds\f{1}{|\Omega|}\int_\Omega m_1(x)dx\;\; \text{in}\;\;C^2(\overline \Omega),$$
which yields
\begin{equation}\label{sti}
\lim_{\delta\to0} \theta_{1,\delta}=0\;\; \text{in}\;\;C^2(\overline \Omega).
\end{equation}
Denote
\begin{equation*}f_1(\delta):=\mu_1(d_2,m_2-b\theta_{1,\delta})\;\;\text{and}\;\; f_2(\delta):=\mu_1(d_1,\delta m_1(x)-c\theta_2),
 \end{equation*} where $\mu_1(d,w)$ is the principal eigenvalue of \eqref{seigen}. As in the proof of Theorem \ref{ident}, we see that
 $\theta_{1,\delta_1}<\theta_{1,\delta_2}$ if $\delta_1<\delta_2$, which implies that $f_1(\delta)$ is strictly decreasing and $ f_2(\delta)$ is strictly increasing for $\delta\in(0,1)$.
It follows from Eq. \eqref{sti} that
\begin{equation*}
\lim_{\delta\to0} f_1(\delta)>0\;\;\text{and}\;\;\lim_{\delta\to0} f_2(\delta)<0.
\end{equation*}

The following discussions are divided into four cases.\\
{\bf Case (i).} If $(d_1,d_2)\in (\tilde S_v\cup\tilde S_{v,0}\setminus \tilde S_{0,0})\cup \tilde S_{0,0}$, then $$\lim_{\delta\to1} f_1(\delta)\ge 0\;\;\text{and}\;\; \lim_{\delta\to1}  f_2(\delta)\le0.$$
This implies that
$f_1(\delta)>0$ and $f_2(\delta)<0$ for any $\delta\in(0,1)$. It follows from Theorem \ref{clarifi2} that
semitrivial steady state $\left(0,\theta_{d_2,0,0,m_2}\right)$ is globally asymptotically stable for any $\gamma_1,\tau_1>0$.\\
{\bf Case (ii).} If  $(d_1,d_2)\in \tilde S_{-}\cup\left(\tilde S_{u,0}\setminus \tilde S_{0,0}\right)$, then
$$\lim_{\delta\to1} f_1(\delta)\ge0\;\;\text{and}\;\; \lim_{\delta\to1}  f_2(\delta)> 0.$$
Consequently, $f_1(\delta)>0$ for any $\delta\in(0,1)$, and there exists $\tilde \delta \in(0,1)$ such that $f_2(\tilde \delta)=0$,
$f_2(\delta)<0$ for $\delta\in(0,\tilde \delta)$ and $f_2(\delta)>0$ for $\delta\in(\tilde \delta, 1)$. It follows from Theorem \ref{clarifi2} that
the semitrivial steady state
$\left(0,\theta_{d_2,0,0,m_2}\right)$ is globally asymptotically stable for $\gamma_1\tau_1\ge-\ln \tilde\delta$, and for $0<\gamma_1\tau_1<-\ln \tilde\delta$, system \eqref{lvs00} has a unique positive steady state, which is globally asymptotically stable.\\
{\bf Case (iii).} If $(d_1,d_2)\in \tilde S_{u}$, then
$$\lim_{\delta\to1} f_1(\delta)<0\;\;\text{and}\;\; \lim_{\delta\to1}  f_2(\delta)> 0.$$
Consequently, there exist a unique $\delta_2\in(0,1)$ such that $f_1(\delta_2)=0$, and a unique $\delta_1\in (0,1)$ such that $f_2(\delta_1)=0$.
We claim that $\delta_1\le\delta_2$. If it is not true, then $\delta_2<\delta_1$ and $f_1(\delta),f_2(\delta)<0$ for $\delta\in(\delta_2,\delta_1)$, which implies that for the above given $d_1,d_2$,
$$\{(d_1,d_2,\tau_1,0,\gamma_1,0):-\ln \delta_1<\tau_1\gamma_1<-\ln\delta_2\}\subset S_u\cap S_v.$$
This contradicts with the fact
$$\left(S_u\cup S_{u,0}\setminus S_{0,0}\right)\cap\left(S_v\cup S_{v,0}\setminus S_{0,0}\right)=\emptyset.$$
Then if $\delta_1<\delta_2 $, we have $f_1(\delta)>0$ and $f_2(\delta)<0$ for $\delta\in(0,\delta_1)$, $f_1(\delta),f_2(\delta)>0$ for
$\delta\in (\delta_1,\delta_2)$, and $f_1(\delta)<0$ and $f_2(\delta)>0$ for $\delta\in(\delta_2,1)$. Moreover, if $\delta_1=\delta_2$, then
$f_1(\delta)>0$ and $f_2(\delta)<0$ for $\delta\in(0,\delta_1)$, $f_1(\delta)<0$ and $f_2(\delta)>0$ for $\delta\in(\delta_1,1)$, and $f_1(\delta)=f_2(\delta)=0$ for $\delta=\delta_1=\delta_2$. Therefore, $(iii_1)$ and $(iii_2)$ can be obtained directly from Theorem \ref{clarifi2}.
\end{proof}
\begin{remark}
We remark that some of sets $\tilde S_u$, $\tilde S_v$, $\tilde S_{-}$, $\tilde S_{u,0}$, $\tilde S_{v,0}$, $\tilde S_{0,0}$ may be empty for differently chosen parameters $b$ and $c$, and the exact description for these sets could be found in \cite[Theorem 1.4]{HeNi2016-1}.
\end{remark}
It follow from \cite[Theorem 1.3]{HeNi2016-1} that when $\tau_1=\gamma_1=0$, there may exist
four mutually disjoint regions of $(d_1,d_2)$ (see Eq. \eqref{td1d2}), where different global dynamics of model \eqref{lvs00} could occur. However, our results in Theorem \ref{ex2}
imply that
a large delay will lead to the extinction of species $u$ for any $d_1$ and $d_2$.

\subsection{Discussion}
In this subsection, we show briefly that the above method for model \eqref{lvs} can also be applied to the following model:
\begin{equation}\label{lvharm}
\begin{cases}
  \ds\frac{\partial U}{\partial t}=d_1\Delta U+U\left[m_1(x)-U-cV(x,t-\tau_2)\right], & x\in \Omega,\; t>0,\\
 \ds\frac{\partial V}{\partial t}=d_2\Delta V+V\left[m_2(x)-bU(x,t-\tau_1)-V\right], & x\in\Omega,\; t>0,\\
\ds\f{\partial U}{\partial n}=\ds\f{\partial V}{\partial n}=0,& x\in\partial \Omega,\ t>0,\\
U(x,t)=U_0(x,t)\ge0, \;\;x\in\Omega,\;t\in [-\tau_1,0],\\
V(x,t)=V_0(x,t)\ge0,\;\;x\in\Omega,\;t\in [-\tau_2,0].\\
\end{cases}
\end{equation}
The global dynamics and traveling waves of model \eqref{lvharm} were studied extensively for the homogeneous case (i.e., $m_1(x)$ and $m_2(x)$ are constant), see \cite{GourleyRuan2003,LinLi,LinLi2,LvWang} and references therein.
By virtue of the similar arguments as in the proof of Proposition \ref{monot}, we see that model \eqref{lvharm} also generates a monotone dynamical system.
\begin{proposition}\label{monot2}
Let $(U_i(x,t),V_i(x,t))$ be the corresponding solution of model \eqref{lvharm} with initial value $(U_{0,i},V_{0,i})$ for $i=1,2$. Assume that
\begin{equation*}
\begin{split}
&U_{0,1}\ge U_{0,2}\ge0 \;\;\text{for}\;\;x\in\overline\Omega,\;t\in[-\tau_1,0],\\
&0\le V_{0,1}\le V_{0,2}\;\;\text{for}\;\;x\in\overline\Omega,\;t\in[-\tau_2,0].
 \end{split}
 \end{equation*}
Then $$U_1(x,t)\ge U_2(x,t)\;\;\text{and}\;\; V_1(x,t)\le V_2(x,t)\;\;\text{for}\;\;x\in\overline\Omega,\;t\ge0.$$
\end{proposition}
Letting
$(u,v)$ be the positive steady state of system \eqref{lvharm}, and linearizing system \eqref{lvharm} at $(u,v)$, one could obtain the following eigenvalue problem
\begin{equation}\label{eigharm}
\begin{cases}
  \nu\phi_1 =d_1\Delta \phi_1 +m_1(x)\phi_1 -(2u+cv)\phi_1-cu\phi_2e^{-\nu\tau_2}, & x\in \Omega,\\
 \nu\phi_2=d_2\Delta \phi_2+ m_2(x)\phi_2-(bu+2v)\phi_2-bv\phi_1 e^{-\nu\tau_1}, & x\in\Omega,\\
\ds\f{\partial \phi_1}{\partial n}=\ds\f{\partial \phi_2}{\partial n}=0,& x\in\partial \Omega.\\
\end{cases}
\end{equation}
By virtue of the transforation $\psi_1=\phi_1$ and $\psi_2=-\phi_2$, eigenvalue problem \eqref{eigharm} is equivalent to
\begin{equation}\label{eigeharm}
\begin{cases}
  \nu\psi_1 =d_1\Delta \psi_1 +m_1(x)\psi_1 -(2u+cv)\psi_1+cu\psi_2e^{-\nu\tau_2}, & x\in \Omega,\\
 \nu\psi_2=d_2\Delta \psi_2+m_2(x)\psi_2-(bu+2v)\psi_2+bv\psi_1  e^{-\nu\tau_1}, & x\in\Omega,\\
\ds\f{\partial \psi_1}{\partial n}=\ds\f{\partial \psi_2}{\partial n}=0,& x\in\partial \Omega.\\
\end{cases}
\end{equation}
Denote by $\nu_1$ the principal eigenvalue of the following eigenvalue problem
\begin{equation}\label{eige0harm}
\begin{cases}
  \nu\psi_1 =d_1\Delta \psi_1 +m_1(x)\psi_1 -(2u+cv)\psi_1+cu\psi_2, & x\in \Omega,\\
 \nu\psi_2=d_2\Delta \psi_2+m_2(x)\psi_2-(bu+2v)\psi_2+bv\psi_1  , & x\in\Omega,\\
\ds\f{\partial \psi_1}{\partial n}=\ds\f{\partial \psi_2}{\partial n}=0,& x\in\partial \Omega.\\
\end{cases}
\end{equation}
Then we show that eigenvalue problem \eqref{eigharm} (or equivalently, \eqref{eigeharm}) has a principal eigenvalue $\tilde \nu_1$, which has the same sign as
$\nu_1$. \begin{proposition}\label{maiharm}
Assume that $m_i(x)$ satisfies assumption $\mathbf{(M^+)}$ for $i=1,2$, and $d_1,d_2>0$ and $\tau_1,\tau_2\ge0$.
Then there exists a principal eigenvalue $\tilde\nu_1$ of \eqref{eigeharm} associated with the eigenfunction $(\psi_1,\psi_2)>(0,0)$.
Furthermore, $\tilde \nu_1$ has the same sign as $\nu_1$, where $\nu_1$ is the principal eigenvalue of \eqref{eige0harm}, and
$$\tilde\nu_1=\sup\{{\mathcal R}e \nu:\nu \;\;\text{is an eigenvalue of}\;\;\eqref{eigeharm}\}.$$
\end{proposition}
\begin{proof}
For the case that at least one of $\tau_1$ and $\tau$ are positive, we define $\tilde L=(\tilde L_1,\tilde L_2):E\to Y\times Y$ by
\begin{equation*}
\begin{split}
\tilde L_1=&m_1(x)\psi_1(0)+cu\psi_2(-\tau_2),\\
\tilde L_2=&m_2(x)\psi_2(0)+bv\psi_1(-\tau_1),\;\;(\psi_1,\psi_2)\in E,
\end{split}
\end{equation*}
and $\tilde B:=B$, where $B$ is an operator defined in \eqref{B}. Clearly, $\tilde L$ and $\tilde B$ have the same properties as $L$ and $B$, where $L$ is defined
in Eq. \eqref{L}. By the similar arguments as in the proof of Theorem \ref{mai}, we could obtain the desired results.
\end{proof}
Therefore, we see that delays are harmless for model \eqref{lvharm}.
\begin{proposition}
Assume that $m_i(x)$ satisfies assumption $\mathbf{(M^+)}$ for $i=1,2$, and $0<bc\le1$. Then the global dynamics of model
\eqref{lvharm} for $\tau_1,\tau_2>0$ is the same as that for $\tau_1=\tau_2=0$.
\end{proposition}

\end{document}